\documentclass[10pt,a4paper,twoside,reqno,nonamelimits]{amsart}
\usepackage[pagebackref,colorlinks=true,urlcolor=blue,linkcolor=blue,citecolor=blue]{hyperref}
\usepackage{times}
\usepackage{amsmath}
\usepackage{amssymb}
\usepackage{amsfonts}
\usepackage{amscd}
\usepackage{amsthm}
\usepackage{amsxtra}
\usepackage{stmaryrd}
\usepackage[all]{xy}  \CompileMatrices
\usepackage{nicefrac}
\usepackage{bbold}
\usepackage{graphicx}
\usepackage{color}
\usepackage{enumerate}
\usepackage{wrapfig}
\usepackage{url}

\theoremstyle{plain}
\newtheorem{theorem}{Theorem}

\newtheorem*{proposition*}{Proposition}

\newtheorem*{corollary*}{Corollary}
\newtheorem{lemma}[theorem]{Lemma}

\newtheorem*{theorem*}{Theorem}
\newtheorem*{lemma*}{Lemma}
\newtheorem*{conjecture*}{Conjecture}
\newtheorem{conjecture}[theorem]{Conjecture}
\newtheorem*{question*}{Question}
\newtheorem{question}[theorem]{Question}
\newtheorem*{problem*}{Problem}

\theoremstyle{definition}

\newtheorem*{exercise*}{Exercise}

\theoremstyle{remark}

\newtheorem*{remark*}{Remark}

\newtheorem*{remarks*}{Remarks}

\newtheorem*{claim*}{Claim}

\newcommand{\subclass}[1]{}

\newcommand{\enumTi}[1]{\renewcommand{\theenumi}{#1}}

\newcommand{\alphenumi}{\enumTi{\alph{enumi}}}

\newcommand{\romenumi}{\enumTi{\roman{enumi}}}

\DeclareMathOperator{\rk}{rk}

\newcommand{\FF}{\mathbb{F}}

\newcommand{\QQ}{\mathbb{Q}}
\newcommand{\RR}{\mathbb{R}}

\newcommand{\ZZ}{\mathbb{Z}}
\newcommand{\kk}{\mathbb{k}}

\newlength{\algotabbingwidth}
\renewcommand{\paragraph}[1]{\par\smallskip\noindent{#1}}

\newcommand{\Mt}{M^{\scriptscriptstyle(t)}}
\newcommand{\nt}{n^{\scriptscriptstyle(t)}}

\begin{document}

\title[Fooling-sets and rank in nonzero characteristic]%
{Fooling-sets and rank in nonzero characteristic\\
  \footnotesize(extended abstract)%
}%


\author{Mirjam Friesen${}^\dagger$}%
\author{Dirk Oliver Theis${}^*$}%

\date{{Fri Mar 15 16:10:39 EET 2013}\\ %
  ${}^\dagger$ Faculty of Mathematics, Otto von Guericke University Magdeburg, Germany\\
  ${}^*$ Faculty of Mathematics and Computer Science, University of Tartu, Estonia,
  \tiny\texttt{dirk.oliver.theis@ut.ee}}

\begin{abstract}
  An $n\times n$ matrix $M$ is called a \textit{fooling-set matrix of size $n$,} if its diagonal entries
  are nonzero, whereas for every $k\ne \ell$ we have $M_{k,\ell} M_{\ell,k} = 0$.  Dietzfelbinger,
  Hromkovi{\v{c}}, and Schnitger (1996) showed that $n \le (\rk M)^2$, regardless of over which field the
  rank is computed, and asked whether the exponent on $\rk M$ can be improved.

  \noindent%
  We settle this question for nonzero characteristic by constructing a family of matrices for which the
  bound is asymptotically tight.  The construction uses linear recurring sequences.
\end{abstract}
\maketitle

\section{Introduction}

An $n\times n$ matrix~$M$ over some field~$\kk$ is called a \textit{fooling-set matrix of size~$n$} if
\begin{subequations}\label{eq:def-fool}
  \begin{eqnarray}
    &M_{kk} \ne 0 &\text{ for all~$k$ (its diagonal entries are all nonzero), and} \label{eq:def-fool:diag}\\
    &M_{k,\ell} \, M_{\ell,k} = 0 &\text{ for all $k\ne \ell$.}                    \label{eq:def-fool:off-diag}
  \end{eqnarray}
\end{subequations}
Note that the definition depends only on the zero-nonzero pattern of~$M$.

In Communication Complexity and Combinatorial Optimization, one is interested in finding a large
fooling-set (sub-)matrix contained in a given matrix~$A$ (permutation of rows and columns is allowed), as
its size provides a lower bound to other numerical properties of the matrix.  Since large fooling-set
submatrices are typically difficult to identify (the problem is equivalent to finding a large clique in a
graph of a certain type), one would like to upper-bound the size of a fooling-set matrix one may possibly
hope for in terms of easily computable properties of~$A$.

Dietzfelbinger, Hromkovi{\v{c}}, and Schnitger (\cite[Thm.~1.4]{DietzfelbingerHromkovicSchnitger96}, or
see~\cite[Lemma~4.15]{KushilevitzNisan97}; cf.~\cite{KlauckDewolf13,FioriniKaibelPashkovichTheis13})
proved that the rank of a fooling-set matrix of size~$n$ is at least $\sqrt n$, i.e.,
\begin{equation}\label{eq:rk-of-fool} 
  n \le (\rk_{\kk} M)^2.
\end{equation}
This inequality gives such an upper bound on the largest fooling-set submatrix in terms of the easily
computable rank of~$A$.

However, it is an open question whether the exponent on the rank in the right-hand side
of~\eqref{eq:rk-of-fool} can be improved or not.  Dietzfelbinger et al.~\cite[Open
Problem~2]{DietzfelbingerHromkovicSchnitger96} were particularly interested in 0/1-matrices and
$\kk=\FF_2$, which corresponds to the Communication Complexity situation they dealt with.

Klauck and de Wolf~\cite{KlauckDewolf13} have pointed out the importance for Communication Complexity of
the question regarding general (i.e., not 0/1) matrices.

Currently, the examples (attributed to M.~H\"uhne in~\cite{DietzfelbingerHromkovicSchnitger96}) of 0/1
fooling-set matrices~$M$ with smallest rank are such that $n \approx (\rk_{\FF_2} M)^{\log_4 6}$ ($\log_4
6 = 1.292\dots$); for general matrices, Klauck and de Wolf~\cite{KlauckDewolf13} have given examples with
$n \approx (\rk_\QQ M)^{\log_3 6}$ ($\log_3 6 = 1.63\dots$).

In our paper, we settle the question for fields~$\kk$ of nonzero characteristic.  We prove that
inequality~\eqref{eq:rk-of-fool} is asymptotically tight if the characteristic of~$\kk$ is nonzero.
Notably, not only is the exponent on the rank in inequality~\eqref{eq:rk-of-fool} best possible, but so
is the constant (one) in front of the rank.


\paragraph{\bf Organization of this extended abstract.}  %
In the next section we will explain some of the connections of the fooling-set vs.\ rank problem with
Combinatorial Optimization and Graph Theory concepts.
In Section~\ref{sec:construction}, we will sketch the proof of our result.
In the final section, we point to some questions which remain open.

\section{Some Remarks on the Importance of Fooling-Set Matrices}\label{sec:connect}

While the fooling-set size vs.\ rank problem is of interest in its own right as a minimum-rank type
problem in Combinatorial Matrix Theory, fooling-set matrices are connected to other areas of Mathematics
and Computer Science.

\small%
\paragraph{\bf In Polytope Theory,} given a polytope~$P$, sizes of fooling-set submatrices of
appropriately defined matrices provide lower bounds to the number of facets of any polytope~$Q$ which can
be mapped onto~$P$ by a projective mapping (\cite{Yannakakis91},
cf.~\cite{FioriniKaibelPashkovichTheis13}).
Similarly, in \textbf{Combinatorial Optimization,} sizes of fooling-set matrices are lower bounds to the
minimum sizes of Linear Programs for combinatorial optimization problems (\cite{Yannakakis91}).  For
example, it is an open question whether Edmond's matching polytope for a complete graph on~$n$ vertices
admits a fooling-set matrix whose size grows quicker in~$n$ than the dimension of the polytope.  Such a
fooling-set matrix would yield a fairly spectacular improvement on the currently known lower bounds of
sizes of Linear Programming formulations for the matching problem.
See~\cite{FioriniKaibelPashkovichTheis13} for bounds based on fooling sets for a number of combinatorial
optimization problems, including bipartite matching.

In the Polytope Theory / Combinatorial Optimization applications, we typically have $\kk=\QQ$, and the
rank of the large matrix~$A$ is known.  However, since the definition of a fooling-set matrix depends
only on the zero-nonzero pattern, changing the field from $\QQ$ to $\kk'$ and replacing the nonzero
rational entries of~$A$ by nonzero numbers in~$\kk'$ may yield a lower rank and hence a better upper
bound on the size of a fooling-set matrix.

\paragraph{\bf In Computational Complexity,} %
fooling-set matrices provide lower bounds for the communication complexity of Boolean functions (see,
e.g.,
\cite{AroraBarak09,KushilevitzNisan97,LovaszSaks88Moeb,DietzfelbingerHromkovicSchnitger96,KlauckDewolf13}),
and for the number of states of an automaton accepting a given language (e.g., \cite{GruberHolzer06}).

\paragraph{\bf In Graph Theory,} %
a fooling-set matrix (up to permutation of rows and columns) can be understood as the incidence matrix of
a bipartite graph containing a perfect cross-free matching.  Recall that a matching in a bipartite
graph~$H$ is called \textit{cross-free} if no two matching edges induce a~$C_4$-subgraph of~$H$.

Cross-free matchings are best known as a lower bound on the size of biclique coverings of graphs (e.g.\
\cite{Dawande03,JuknaKulikov09}).  A \textit{biclique covering} of a graph~$G$ is a collection of
complete bipartite subgraphs of~$G$ such that each edge of~$G$ is contained in at least one of these
bipartite subgraphs.  If a cross-free matching of size~$n$ is contained as a subgraph in~$G$, then at
least~$n$ bicliques are needed to cover all edges of~$G$.  (For some classes of graphs, this is a sharp
lower bound on the biclique covering number~\cite{Dawande03,SotoTelha11}).

\paragraph{\bf In Matrix Theory,} the maximum size of a fooling-set sub-matrix is known under a couple of
different names, e.g.~as independence number \cite[Lemma 2.4]{CohenRothblum93}), or as the intersection
number.  For some semirings, this number provides a lower bound for the so-called factorization rank of
the matrix over the semiring.


\normalsize%
\paragraph{\bf In each of these areas,}  %
fooling-set matrices are used as lower bounds.  Upon embarking on a search for a big fooling-set matrix
in a large, complicated matrix~$A$, one is interested in an \textit{a priori} upper bound on their sizes
and thus the potential usefulness of the lower bound method.

\numberwithin{theorem}{section}
\section{Fooling-Set Matrices from Linear Recurring Sequences}\label{sec:construction}

For a prime number~$p$, we denote by $\FF_p$ the finite field with~$p$ elements.  The following is an
accurate statement of our result.

\begin{theorem}\label{thm:main}
  For every prime number~$p$, there is a family of fooling-set matrices $\Mt$ over $\FF_p$ of size~$\nt$,
  $t=1,2,3,\dots$, such that $\nt \to \infty$, and
  \begin{equation*}
    \frac{ \nt }{ (\rk_{\FF_p} \Mt)^2 } \;\longrightarrow 1.
  \end{equation*}
\end{theorem}

The method used in all the \emph{earlier} examples (mentioned in the introduction) of fooling-set
matrices with small rank was the following: One conjures up a single, small fooling-set matrix~$M^0$ (of
size, say, 6), determines its rank (say, 3), and then uses the tensor-powers of~$M^0$ (which are
fooling-set matrices, too).  With these numerical values, from~$M^0$, one obtains $\log_36$ as a lower
bound on the exponent on the rank in~\eqref{eq:rk-of-fool}.

Our technique is a departure from that approach.  As noted above, we use linear recurring sequences.  For
every~$t$, we construct an $\nt$-periodic function, which gives us a fooling-set matrix of size~$\nt$.

\paragraph{We now describe that construction.}  %
Let~$p$ be a prime number and $r \ge 2$ an integer.  Define the function $f\colon \ZZ\to \FF_p$ by the
recurrence relation
\begin{subequations}\label{eq:def-f}
  \begin{equation}\label{eq:def-f:recrel}
    f(k+r) = -f(k) - f(k+1) \quad\text{for all $k\in \ZZ$}
  \end{equation}
  and the initial conditions
  \begin{equation}\label{eq:def-f:initial}
    f(0) = 1\text{, and } f(1) = \ldots = f(r-1) = 0.
  \end{equation}
\end{subequations}

Fix an integer $n > r$.  From the sequence, we define an $n\times n$ matrix as follows.  For ease of notation, the matrix indices are taken to be in
$\{0,\dots,n-1\}\times \{0,\dots,n-1\}$.  We let
\begin{equation}\label{eq:def-M}
  M_{k,\ell} := f(k-\ell).
\end{equation}

It is fairly easy to see that $\rk M \le r$.

\begin{lemma}\label{lem:rk-M}
  The rank of~$M$ is at most~$r$.
\end{lemma}
\begin{proof}
  From~\eqref{eq:def-f:recrel}, for $k \ge r$, we deduce the equation $M_{k,\star} = -M_{k-r,\star} - M_{k-r+1,\star}$.
  Hence, each of the rows $M_{k,\star}$, $k \ge r$, is a linear combination of the first~$r$ rows of~$M$.
\end{proof}

It can be seen that the rank is, in fact, equal to~$r$: The top-left $r\times r$ sub-matrix is regular
because it is upper-triangular with nonzeros along the diagonal.

Next, we reduce the fooling-set property~\eqref{eq:def-fool} to a property of the function~$f$.

\begin{lemma}\label{lem:fool-eq}\mbox{}
  The matrix~$M$ defined in~\eqref{eq:def-M} is a fooling-set matrix, if and only if,
  \begin{equation}\label{eq:cross-symmetry}
    f(k) f(-k) = 0  \quad\text{ for all $k \in \{1,\dots,n-1\}$.}
  \end{equation}
\end{lemma}
\begin{proof}
  It is clear from \eqref{eq:def-f:initial} and~\eqref{eq:def-M} that $M_{j,j} = f(0) = 1$ for all
  $j=0,\dots,n-1$, so it remains to verify~\eqref{eq:def-fool:off-diag}.
  Since
  \begin{equation*}
    M_{i,j} M_{j,i} = f(i-j) f(j-i) = f(i-j) f(-(i-j)),
  \end{equation*}
  if $f(k) f(-k) = 0$ for all $k=1,\dots,n-1$, then $M_{i,j} M_{j,i}$ is zero whenever $i\ne j$.  This
  proves~\eqref{eq:def-fool:off-diag}.
\end{proof}


Given appropriate conditions on $r$ and~$n$ (depending on~$p$), this condition on~$f$ can indeed be
verified:

\begin{lemma}\label{lem:key-lemma}
  For all integers $t \ge 1$, if we let $r := p^t+1$ and $n := r(r-1)+1$, then $f(k)f(-k) = 0$ for all
  $k\in \ZZ\setminus n\ZZ$.
\end{lemma}

Combining the above three lemmas, we can complete the proof of Theorem~\ref{thm:main}.

\begin{proof}[Proof of Theorem~\ref{thm:main}.]
  Let~$p$ be a prime number.
  For every integer $t\ge 1$, let $r := p^t+1$ and $\nt := r(r-1) +1$, and define the matrix $\Mt := M$
  over $\FF_p$ as in~\eqref{eq:def-M}.  
  By Lemma~\ref{lem:rk-M}, the rank of $\Mt$ is at most~$r$, and from Lemmas \ref{lem:fool-eq}
  and~\ref{lem:key-lemma} we conclude that $\Mt$ is a fooling-set matrix.  Hence, we have
  \begin{equation*}
    1 \ge \frac{ \nt }{ \rk_{\FF_p} (\Mt)^2 } \ge \frac{ r^2-r+1 }{r^2} \ge 1 - p^{-t}/4 \xrightarrow{t\to\infty} 1,
  \end{equation*}
  where the left-most inequality is from~\eqref{eq:rk-of-fool}.
\end{proof}

To prove Lemma~\ref{lem:key-lemma}, we need two more lemmas.  The first one states that in every section
$\{jr,\dots,(j+1)r-1\}$, $j=0,1,\dots$, there is a block of zeros whose length decreases with~$j$.

\begin{lemma}\label{lem:zero-blocks}
  For $j=0,\dots,r-2$, we have
  \begin{equation}\label{eq:zero-block}
    f(jr + i ) = 0 \quad\text{for $i = 1,\dots, r-1-j$.}
  \end{equation}
\end{lemma}
\begin{proof}
  Equation~\eqref{eq:zero-block} is true for $j=0$ by~\eqref{eq:def-f:initial}.
  Suppose~\eqref{eq:zero-block} holds for some $j<r-2$.  Then $f((j+1)r + i ) = 0$ for $i = 1,\dots, r-1-(j+1)$, because, by~\eqref{eq:def-f:recrel},
  \begin{equation*}
    f((j+1)r + i)
    =
    f(jr + i + r)
    = 
    -f(jr + i)  -  f(jr + (i+1))
    = -0 - 0
  \end{equation*}
  holds.
\end{proof}

Every function on~$\ZZ$ with values in a finite field which is defined by a (reversible) linear
recurrence relation is periodic (cf.\ e.g.~\cite{LidlNiederreiter94}).  The second lemma establishes that
a specific number~$n$ is a period of~$f$ as defined in~\eqref{eq:def-f}.

\begin{lemma}\label{lem:periodicity}
  If $r = p^t+1$ for some integer $t\ge1$, then $n := r(r-1) +1$ is a period of the function~$f$.
\end{lemma}

This lemma is the difficult part of the proof of Theorem~\ref{thm:main}.  Due to the space limitations,
for its proof, we have to refer to the full paper.  At this point, suffice it to say that the argument
proceeds by identifying binomial coefficients among the values of~$f$, and then uses the known
periodicity of the binomial coefficients modulo~$p$.

Lemmas \ref{lem:zero-blocks} and~\ref{lem:periodicity} allow us to prove Lemma~\ref{lem:key-lemma}.

\begin{proof}[Proof of Lemma~\ref{lem:key-lemma}.]
  We need to show $f(k) f(-k) = 0$ whenever $n \nmid k$.  By Lemma~\ref{lem:periodicity}, this is
  equivalent to showing $f(k) f(n-k) = 0$ for $k=1,\dots,n-1$.  Given such a~$k$, let $j,i$ be such that
  $k = jr +i$ and $0\le i \le r-1$.
  
  If $i \le r-1-j$, then $f(k)=0$ by Lemma~\ref{lem:zero-blocks}, and we are done.  If, on the other
  hand, $i > r-1-j$, then
  \begin{equation*}
    n-k
    =
    r^2 - r +1 - jr - i
    =
    (r-1-(j+1))r  + (r-i+1),
  \end{equation*}
  and $r-i+1 \le j+1$, so, by Lemma~\ref{lem:zero-blocks}, we have $f(n-k) = 0$.
\end{proof}

\section{Conclusion}\label{sec:conclusio}

Dietzfelbinger et al.'s original question regarding the tightness of inequality~\eqref{eq:rk-of-fool} for
0/1-matrices remains open in characteristic $p > 2$.  For these matrices, it may still be possible that
the exponent on the rank in the inequality~\eqref{eq:rk-of-fool} can be improved.

For characteristic zero, Klauck and de Wolf~\cite{KlauckDewolf13} have given an example of a fooling-set
matrix of size~$6$ with entries in $\{0,\pm1\}$ which as rank~$3$ .  Thus, using the method sketched
above (following Theorem~\ref{thm:main}), the exponent on the rank in inequality~\eqref{eq:rk-of-fool}
with $\kk:=\QQ$ for general (i.e., not 0/1) matrices is at least $\log_3 6 = 1.63\dots$, while the best
known bound for 0/1-matrices is $\log_4 6 = 1.292\dots$.

We would like to point out the possibility that, in characteristic zero, the minimum achievable rank on
the right hand side of inequality~\eqref{eq:rk-of-fool} may depend not only on the characteristic, but on
the field~$\kk$ itself.  Indeed, there are examples of zero-nonzero patterns for which the minimum rank
of a matrix with that zero-nonzero pattern differs between $\kk = \QQ$ and $\kk = \RR$, see
e.g.~\cite{KoppartyBhaskararao08}.
Hence, for characteristic zero, we ask the following weaker version of Dietzfelbinger et al.'s question.

\begin{question}
  Is there a field~$\kk$ (of characteristic zero) over which the fooling-set matrix size vs.\ rank
  inequality in~\eqref{eq:rk-of-fool} can be improved?
\end{question}

As mentioned in Section~\ref{sec:connect}, another question in characteristic zero comes from polytope
theory.  Let~$P$ be a polytope.  Let $A$ be a matrix whose rows are indexed by the facets of~$P$ and
whose columns are indexed by the vertices of~$P$, and which satisfies
$A_{F,v} = 0$, if $v \in F$, and $A_{F,v} \ne 0$, if $v \notin F$.
For any fooling-set submatrix of size~$n$ of~$A$, the following inequality follows
from~\eqref{eq:rk-of-fool} (cf.~\cite{FioriniKaibelPashkovichTheis13}):
\begin{equation}\label{eq:dietz:ptp}
  n \le (\dim P+1)^2.
\end{equation}
The following variant of Dietzfelbinger et al.'s question is of pertinence in Polytope Theory and
Combinatorial Optimization (see Section~\ref{sec:connect}).

\begin{question}
  Can the fooling-set size vs.\ dimension inequality~\eqref{eq:dietz:ptp} be improved (for polytopes)?
\end{question}

To our knowledge, the best known lower bound for the best possible exponent on the dimension in
inequality~\eqref{eq:dietz:ptp} is~1.

Finally, the complexity of the Fooling-Set-Submatrix problem is still open:

\begin{conjecture}
  The \textit{Fooling-Set-Submatrix problem}\\
  \begin{tabular}{ll}
    \bf Input:&  Integers~$n,m$ and $m\times m$ 0/1-matrix~$A$ \\
    \bf Output:& ``Yes'', if a fooling-set submatrix of size~$n$ of~$A$ exists,\\&``No'' otherwise.
  \end{tabular}\\
  is NP-hard.
\end{conjecture}

\tiny%

\providecommand{\bysame}{\leavevmode\hbox to3em{\hrulefill}\thinspace}
\providecommand{\MR}{\relax\ifhmode\unskip\space\fi MR }
\providecommand{\MRhref}[2]{%
  \href{http://www.ams.org/mathscinet-getitem?mr=#1}{#2}
}
\providecommand{\href}[2]{#2}


\begin{thebibliography}{10}

\bibitem{AroraBarak09}
Sanjeev Arora and Boaz Barak, \emph{Computational complexity}, Cambridge
  University Press, Cambridge, 2009, A modern approach. \MR{2500087
  (2010i:68001)}

\bibitem{CohenRothblum93}
Joel~E. Cohen and Uriel~G. Rothblum, \emph{Nonnegative ranks, decompositions,
  and factorizations of nonnegative matrices}, Linear Algebra Appl.
  \textbf{190} (1993), 149--168. \MR{1230356 (94i:15015)}

\bibitem{Dawande03}
Milind Dawande, \emph{A notion of cross-perfect bipartite graphs}, Inform.
  Process. Lett. \textbf{88} (2003), no.~4, 143--147. \MR{2009283
  (2004g:05118)}

\bibitem{DietzfelbingerHromkovicSchnitger96}
Martin Dietzfelbinger, Juraj Hromkovi{\v{c}}, and Georg Schnitger, \emph{A
  comparison of two lower-bound methods for communication complexity}, Theoret.
  Comput. Sci. \textbf{168} (1996), no.~1, 39--51, 19th International Symposium
  on Mathematical Foundations of Computer Science (Ko{\v{s}}ice, 1994).
  \MR{1424992 (98a:68068)}

\bibitem{FioriniKaibelPashkovichTheis13}
Samuel Fiorini, Volker Kaibel, Kanstantin Pashkovich, and Dirk~Oliver Theis,
  \emph{Combinatorial bounds on nonnegative rank and extended formulations},
  \href{http://arxiv.org/abs/1111.0444}{arXiv:1111.0444}) (to appear in
  \textit{Discrete Math.}), 2013+.

\bibitem{GruberHolzer06}
Hermann Gruber and Markus Holzer, \emph{Finding lower bounds for
  nondeterministic state complexity is hard (extended abstract)}, Developments
  in language theory, Lecture Notes in Comput. Sci., vol. 4036, Springer,
  Berlin, 2006, pp.~363--374. \MR{2334484}

\bibitem{JuknaKulikov09}
S.~Jukna and A.~S. Kulikov, \emph{On covering graphs by complete bipartite
  subgraphs}, Discrete Math. \textbf{309} (2009), no.~10, 3399--3403.
  \MR{2526759 (2010h:05231)}

\bibitem{KlauckDewolf13}
Hartmut Klauck and Ronald de~Wolf, \emph{Fooling one-sided quantum protocols},
  \href{http://arxiv.org/abs/1204.4619}{arXiv:1204.4619}, 2012.

\bibitem{KoppartyBhaskararao08}
Swastik Kopparty and K.~P.~S. Bhaskara~Rao, \emph{The minimum rank problem: a
  counterexample}, Linear Algebra Appl. \textbf{428} (2008), no.~7, 1761--1765.
  \MR{2388655 (2009a:15002)}

\bibitem{KushilevitzNisan97}
Eyal Kushilevitz and Noam Nisan, \emph{Communication complexity}, Cambridge
  University Press, Cambridge, 1997. \MR{1426129 (98c:68074)}

\bibitem{LidlNiederreiter94}
Rudolf Lidl and Harald Niederreiter, \emph{Introduction to finite fields and
  their applications}, first ed., Cambridge University Press, Cambridge, 1994.
  \MR{1294139 (95f:11098)}

\bibitem{LovaszSaks88Moeb}
L.~Lov{\'a}sz and M.~Saks, \emph{M{\"o}bius functions and communication
  complexity}, Proc.\ 29th IEEE FOCS, IEEE, 1988, pp.~81--90.

\bibitem{SotoTelha11}
Jos{\'e}~A. Soto and Claudio Telha, \emph{Jump number of two-directional
  orthogonal ray graphs}, Integer programming and combinatorial optimization,
  Lecture Notes in Comput. Sci., vol. 6655, Springer, Heidelberg, 2011,
  pp.~389--403. \MR{2820923 (2012j:05305)}

\bibitem{Yannakakis91}
Mihalis Yannakakis, \emph{Expressing combinatorial optimization problems by
  linear programs}, J. Comput. System Sci. \textbf{43} (1991), no.~3, 441--466.
  \MR{1135472 (93a:90054)}

\end{thebibliography}
\end{document}